\documentclass[11pt]{article}

\usepackage[a4paper,left=2cm,right=2cm,top=2.5cm,bottom=2.5cm]{geometry}
\usepackage{amsmath,amssymb,amsfonts,amsthm}
\usepackage{mathrsfs}
\usepackage{latexsym}
\usepackage{enumerate}
\usepackage{verbatim}
\usepackage{comment}
\usepackage{color}
\usepackage{graphicx}
\usepackage{float}
\usepackage{booktabs}

\usepackage{algorithm}
\usepackage{algorithmic}

\usepackage{epsfig}
\usepackage{epsf}
\usepackage{graphpap}
\usepackage{epic}
\usepackage{curves}
\usepackage{tikz}
\usepackage{subfigure}
\usepackage{pstricks}

\usepackage[numbers]{natbib}
\usepackage{hyperref}

\usepackage{authblk}

\definecolor{gray}{rgb}{0.25, 0.25, 0.25}

\usepackage{titlesec}

\renewcommand{\thesection}{\arabic{section}}
\renewcommand{\thesubsection}{\thesection.\arabic{subsection}}
\renewcommand{\thesubsubsection}{\thesubsection.\arabic{subsubsection}}

\titleformat{\section}
  {\normalfont\Large\bfseries}
  {\thesection.}
  {0.5em}
  {}

\titleformat{\subsection}
  {\normalfont\large\bfseries}
  {\thesubsection.}
  {0.5em}
  {}

\titleformat{\subsubsection}
  {\normalfont\normalsize\bfseries}
  {\thesubsubsection.}
  {0.5em}
  {}

\numberwithin{equation}{section}

\def \[{\begin{equation}}
\def \]{\end{equation}}

\newtheorem{thm}{Theorem}[section]
\newtheorem{defi}[thm]{Definition}
\newtheorem{lem}[thm]{Lemma}
\newtheorem{cor}[thm]{Corollary}
\newtheorem{prop}[thm]{Proposition}

\newtheorem{con}[thm]{Conjecture}

\newtheorem*{thm*}{Theorem}
\newtheorem*{prop*}{Proposition}

\newcommand\ex{\ensuremath{\mathrm{ex}}}

\newcommand\cA{{\mathcal A}}

\newcommand\cD{{\mathcal D}}

\newcommand\cF{{\mathcal F}}
\newcommand\cG{{\mathcal G}}
\newcommand\cH{{\mathcal H}}

\newcommand\cS{{\mathcal S}}

\begin{document}
\title{
{A Spectral Confirmation of the Erd\H{o}s Matching Conjecture} 
}

\author{Liying Kang\,}
\author{Yongchun Lu$^*$\,}
\author{Xiying Yuan\,}
\author{Junpeng Zhou\,}

\affil{\small \,Department of Mathematics, Shanghai University, Shanghai 200444, PR China \\
\,Newtouch Center for Mathematics of Shanghai University, Shanghai 200444, PR China}

\date{}

\maketitle

\footnotetext{*\textit{Corresponding author}.}
\footnotetext{
Email addresses: \texttt{lykang@shu.edu.cn} (L. Kang), \texttt{luyongchun@shu.edu.cn} (Y. Lu),
\texttt{xiyingyuan@shu.edu.cn} (X. Yuan), \texttt{junpengzhou@shu.edu.cn} (J. Zhou).}

\begin{abstract}
The Erd\H{o}s Matching Conjecture concerns the maximum number of hyperedges in an $r$-uniform hypergraph with bounded matching number. In this paper, we study a spectral counterpart of this conjecture. For sufficiently large $n$, we determine the maximum spectral radius over all $n$-vertex $r$-uniform hypergraphs whose matching number is less than $s$, and characterize the unique extremal hypergraph.

To establish the main theorem, we first apply the shifting method  to reduce the problem to shifted hypergraphs. We then derive several spectral upper bounds through hypergraph decomposition and related variational estimates for tensor spectral radii. 
With these estimates, we analyze the structural properties of shifted-saturated hypergraphs and prove the spectral extremal theorem for shifted hypergraphs with bounded matching numbers. Finally, we drop the shifted condition and extend our spectral bound to general  $r$-uniform hypergraphs.

Our main theorem states that for any $n$-vertex $r$-uniform hypergraph $H$ with matching number $\nu(H)<s$, the inequality $\rho(H)\leq \rho(\mathcal{F}_{s-1}(n))$ holds whenever $n$ is sufficiently large. Here $\mathcal{F}_{a}(n)$ denotes the family of all $r$-subsets of $[n]$ intersecting the vertex set $[a]$, and equality is attained if and only if $H$ is isomorphic to $\mathcal{F}_{s-1}(n)$. As an immediate corollary, we derive a spectral counterpart of the classical Erd\H{o}s-Ko-Rado theorem for intersecting hypergraph families.
\end{abstract}

\begin{flushleft}
		\hspace{2.5em}\textbf{Keywords:} Hypergraph, Spectral radius, Erd\H{o}s Matching Conjecture, Shifting\\
		\hspace{2.5em}\textbf{AMS (2000) subject classification:} 05C50, 05C65
	\end{flushleft}

\section{Introduction}
An \textit{$r$-uniform hypergraph} ($r$-graph for short) \scalebox{0.96}{$\cH=(V(\cH),E(\cH))$} consists of a vertex set \scalebox{0.96}{$V(\cH)$} and a hyperedge set \scalebox{0.96}{$E(\cH)$}, where each hyperedge in \scalebox{0.96}{$E(\cH)$} is an $r$-subset of $V(\cH)$. The size of \scalebox{0.96}{$E(\cH)$} is denoted by $e(\cH)$.
For a vertex \scalebox{0.96}{$v \in V(\cH)$}, let \scalebox{0.96}{$E_v(\cH) = \{e \in E(\cH) : v \in e\}$} denote the set of hyperedges containing $v$.
Throughout this paper, we identify hypergraphs with their hyperedge sets. Let \scalebox{0.96}{$\binom{V}{r}$} denote the family of all $r$-subsets of $V$, namely, the complete $r$-graph on $V$. For positive integers $n\geq m$, we use $[n]$ and $[m,n]$ for the sets $\{1,2,\dots,n\}$ and $\{m,m+1,\dots,n\}$, respectively.

Given an $r$-graph $\mathcal{F}$, an $r$-graph $\cH$ is called \textit{$\mathcal{F}$-free} if $\cH$ does not contain a copy of $\mathcal{F}$ as a subhypergraph. The \textit{Tur\'{a}n number} of $\mathcal{F}$, denoted by ${\rm{ex}}_r(n,\mathcal{F})$, is the maximum number of hyperedges among all $n$-vertex $\mathcal{F}$-free $r$-graphs.
A classical result in extremal graph theory is the Erd\H{o}s--Gallai theorem \cite{EGa}, which determines the Tur\'{a}n number of $M_{s}$, where $M_{s}$ denotes a matching of size $s$, i.e., the graph consisting of $s$ independent edges. 
It is natural to consider the corresponding problem for hypergraphs.  Let $M^r_{s}$ denote a matching of size $s$ in an $r$-graph. In 1965, Erd\H{o}s \cite{Er1} proposed the following conjecture on hypergraph  matchings.
\medskip
\begin{con}[Erd\H{o}s Matching Conjecture~\cite{Er1}]\label{con1}
Let integers $s,r\geq2$ and $n\geq sr-1$. Then $$\ex_r(n,M_{s}^r)\leq \max\left\{ \binom{sr-1}{r}, \binom{n}{r} - \binom{n-s+1}{r} \right\}.$$
\end{con}
\medskip
The above conjecture is known to hold for the cases $r=2$ \cite{EGa} and $r=3$ \cite{Fran2017,LuMi}. For the general case, Erd\H{o}s \cite{Er1} verified Conjecture \ref{con1} for sufficiently large $n$. Later, this threshold was improved by Bollob\'as, Daykin and Erd\H{o}s \cite{68}, Huang, Loh and Sudakov \cite{312}, and Frankl, \L uczak and Mieczkowska \cite{215}. For further developments and relevant results concerning this conjecture, one can refer to \cite{Fr1,213,FrKu,KoKu}. Very recently, Kupavskii and Sokolov \cite{KuSo} completely resolved the non-uniform case of the conjecture for $n\leq 3s$. Hou, Hu and Liu \cite{Hou-Hu-Liu} proved the conjecture for the case $r=4$ and $s\geq 6961$.

The goal of this paper is to establish a spectral confirmation of the Erd\H{o}s Matching Conjecture for sufficiently large $n$. We begin by introducing some notation. For positive integers $r$ and $n$, a real tensor $\mathcal{A}=$ ( $a_{i_1 i_2 \cdots i_r}$ ) of order $r$ and dimension $n$ is defined as a multidimensional array where each entry $a_{i_1 i_2 \cdots i_r} \in \mathbb{R}$ for all $i_1, i_2, \ldots, i_r \in[n]$. A tensor $\mathcal{A}$ is called \textit{symmetric} if $\mathcal{A}_{i_1i_2\cdots i_r}=\mathcal{A}_{\sigma({i_1})\cdots\sigma({i_r})}$, where $\sigma$ is any permutation of its indices. The concept of tensor eigenvalues was independently established by Qi \cite{Qi-new} and Lim \cite{Lim} in 2005.
\medskip
\begin{defi}[\cite{Lim}, \cite{Qi-new}]
Let $\mathcal{A}$ be a tensor of order $r$ and dimension $n$. Let ${\bf x}^{[r]}=(x_1^r,\ldots,x_n^r)^{\rm T}$. Then $\mathcal{A}{\bf{x}}^{r-1}$ is defined to be a vector in $\mathbb{C}^n$\ whose $i$-th component is
$$(\mathcal{A}{\bf x}^{r-1})_i
=\sum_{i_2,\dots,i_r=1}^{n}
a_{ii_2\cdots i_r}x_{i_2}\cdots x_{i_r},
\quad i=1,2,\dots,n.$$
If there exists a number $\lambda\in\mathbb{C}$ and a nonzero vector ${\bf x}\in\mathbb{C}^n$ such that
$$\mathcal{A}{\bf x}^{r-1}=\lambda {\bf x}^{[r-1]},$$
then $\lambda$ is called an eigenvalue of $\mathcal{A}$, $\bf{x}$ is called an
eigenvector of $\mathcal{A}$ corresponding to the eigenvalue $\lambda$.
The spectral radius of $\mathcal{A}$ is the maximum modulus of the eigenvalues
of $\mathcal{A}$.
\end{defi}
\medskip
Shao \cite{Sh} defined the general product of tensors, thus $\mathcal{A} {\bf x}^{r-1}$ can be simply written as $\mathcal{A} x$. In the following, we use $\mathcal{A} {\bf x}$ to denote $\mathcal{A} {\bf x}^{r-1}$.
Let $\mathbb{R}^n_+$ and $\mathbb{R}^n_{++}$ denote the sets of nonnegative real vectors and positive real vectors of dimension $n$, respectively.
\medskip
\begin{lem}[\cite{Qi}]\label{lem:Qi}
Let ${\cA}$ be an $r$-order $n$-dimension nonnegative symmetric tensor with integer $r\ge2$. Then its spectral radius admits the variational characterization
\begin{equation*}
\rho({\cA})=\max\left\{\,{\bf x}^\mathrm{T}{\cA}{\bf x}: {\bf x}=(x_1,\dots,x_n)^\mathrm{T}\in\mathbb{R}^n_+,\ \sum_{i=1}^n x_i^r=1\,\right\}.
\end{equation*}
\end{lem}
\medskip
Let $\cA$ be an $r$-order $n$-dimension tensor. If there exists a nonempty proper subset $I\subseteq [n]$ such that ${a}_{i_1i_2\cdots i_r}=0$ for any $i_1\in I$ and $\{i_2,\dots,i_r\}\nsubseteq I$, then ${\cA}$ is called \textit{weakly reducible}. If ${\cA}$ is not weakly reducible, then we say that ${\cA}$ is \textit{weakly irreducible} \cite{Fr-Ga-Ha}.
\medskip
\begin{thm}[\cite{Fr-Ga-Ha},\cite{Ya-Ya}]\label{P-F-thm}
Let ${\cA}$ be an $r$-order $n$-dimension nonnegative tensor.
\begin{enumerate}[\label=(\roman*)]
\item $\rho({\cA})$ is an eigenvalue of ${\cA}$ with a corresponding nonnegative eigenvector.
\item If ${\cA}$ is weakly irreducible, then $\rho({\cA})$ is the unique eigenvalue of ${\cA}$ admitting an eigenvector $\boldsymbol{x}\in\mathbb{R}^n_{++}$, and such eigenvector is unique up to positive scalar multiplication.
\end{enumerate}
\end{thm}
\medskip
 In 2012, Cooper and Dutle \cite{Co-Du} defined the adjacency tensor $\mathcal{A}(\cH)$ for an $r$-graph $\cH$.
\medskip
\begin{defi}[ \cite{Co-Du}]
Let $\cH=(V(\cH),E(\cH))$ be an $r$-graph on $n$ vertices.
The adjacency tensor of $\cH$ is defined as the order $r$ and dimension $n$
tensor $\mathcal{A}(\cH)=(a_{i_1i_2\cdots i_r})$, whose
$(i_1i_2\cdots i_r)$-entry is
\begin{align*}
    a_{i_1i_2\cdots i_r}
=\begin{cases}
\dfrac{1}{(r-1)!}, & \text{if } \{v_{i_1},v_{i_2},\ldots,v_{i_r}\}\in E(\cH),\\[2mm]
0, & \text{otherwise.}
\end{cases}
\end{align*}
\end{defi}

It was proved in \cite{Fr-Ga-Ha} that an $r$-graph $\cH$ is connected if and only if its adjacency tensor $\mathcal{A}(\cH)$ is weakly irreducible. The spectral radius of an $r$-graph $\mathcal{H}$, denoted by $\rho(\mathcal{H})$, is defined as the spectral radius of its adjacency tensor $\mathcal{A}(\mathcal{H})$. If $\mathcal{H}$ is connected, Theorem \ref {P-F-thm} guarantees that $\rho(\mathcal{H})$ is an eigenvalue of $\mathcal{A}(\mathcal{H})$ and possesses the unique positive eigenvector $\mathbf{x}=\left(x_1, \ldots, x_n\right)^{\mathrm{T}}$ with $\|\mathbf{x}\|_r=1$. This vector $\mathbf{x}$ is referred to as the Perron vector of $\mathcal{H}$.  By the eigenequation, the following holds for each $v \in V(\mathcal{H})$:
$$
\rho(\mathcal{H}) x_v^{r-1}=\sum_{e \in E_v(\mathcal{H})} \prod_{u \in e \backslash\{v\}} x_u .
$$
In this paper, we consider the Erd\H{o}s Matching Conjecture in the spectral setting, thereby confirming the Erd\H{o}s Matching Conjecture from a spectral perspective.
Let $\nu(\cH)$ denote the matching number of an $r$-graph $\cH$, i.e., the maximum size of a matching in $\cH$. For an integer $a\ge 0$, define
$$
\mathcal{F}_a(n)=\left\{e\in\binom{[n]}{r}: e\cap [a]\neq\emptyset\right\}.
$$
In particular, $\mathcal{F}_0(n)=\emptyset$. We present our results as follows.
\medskip
\begin{thm}
\label{thm1}
Let $\cH$ be an $r$-graph on $[n]$ with $\nu(\cH)<s$. Then for sufficiently large $n$, we have
$$
\rho(\cH)\le \rho\left(\mathcal{F}_{s-1}(n)\right).
$$
Moreover, equality holds if and only if $\cH\cong \mathcal{F}_{s-1}(n)$.
\end{thm}
\medskip
A family $\cF$ is called \textit{$t$-intersecting} if $|A\cap B| \geq t$ for all $A, B \in \mathcal{F}$. When $t=1$, we say the family is \textit{intersecting}. Recall that the classical Erd\H{o}s--Ko--Rado theorem states that if $\cF\subseteq \binom{[n]}{r}$ is an intersecting family with $n>2r$, then $|\mathcal{F}| \leq \binom{n-1}{r-1}$, with equality only for families isomorphic to $\mathcal{F}_{1}(n)$. As a consequence of Theorem \ref{thm1}, we obtain a spectral analogue of the classical Erd\H{o}s--Ko--Rado theorem.
\medskip
\begin{cor}
\label{cor1}
Let $\mathcal{F} \subseteq \binom{[n]}{r}$ be an intersecting family. Then for sufficiently large $n$, we have
\begin{align*}
\rho(\mathcal{F}) \leq \rho(\mathcal{F}_{1}(n)).
\end{align*}
Moreover, equality holds only for families isomorphic to $\mathcal{F}_{1}(n)$.
\end{cor}
The rest of the paper is organized as follows. Section 2.1 is devoted to the properties of shifting, and Section 2.2 provides several auxiliary inequalities that will be used later. With the preparations in Section 2, we give the complete proof of Theorem \ref{thm1} in Section 3. 

\section{Preliminaries}
In this section, we establish the main tools required for our subsequent proofs. We begin by defining the shifting operation in the first subsection. Subsequently, the second subsection is devoted to several key lemmas concerning spectral radius estimates for hypergraphs.

\subsection{The shifting operation}\label{sec:shifting}
The shifting method was originally introduced by Erd\H{o}s, Ko and Rado \cite{EKR} and further developed by Frankl \cite{Fr111}. It has been widely applied to solve extremal problems in Combinatorics.

Let $\mathcal{H}$ be an $r$-graph on the vertex set $[n]$. For integers $i,j$ with $1 \leq i < j \leq n$ and any $e \in E(\mathcal{H})$, the \textit{shifting operator} $S_{ij}$ is defined by
\begin{align*}
S_{ij}(e) =
\begin{cases}
(e \setminus \{j\}) \cup \{i\}, & \text{if } j \in e, i \notin e \text{ and } (e \setminus \{j\}) \cup \{i\} \notin E(\mathcal{H}); \\[6pt]
e, & \text{otherwise}.
\end{cases}
\end{align*}
Set
\begin{align*}
S_{ij}(\mathcal{H}) = \{ S_{ij}(e) : e \in E(\mathcal{H})\}.
\end{align*}

An $r$-graph $\mathcal{H}$ is called \textit{shifted} (or \textit{stable}) if $S_{ij}(\mathcal{H}) = \mathcal{H}$ for all $1 \leq i < j \leq n$. For two $r$-sets $A=\{a_1<\cdots<a_r\}$ and $B=\{b_1<\cdots<b_r\}$, write $A\preceq B$ if $a_i\le b_i$ for every $i\in[r]$. If an $r$-graph $\cG$ is shifted, then $B\in E(\cG)$ and $A\preceq B$ imply $A\in E(\cG)$.
In \cite{Fr111}, Frankl proved that any $r$-graph $\mathcal{H}$ can be transformed into a shifted $r$-graph by applying the shifting operation iteratively. In the same paper, Frankl also showed that the shifting operation does not increase the matching number of an $r$-graph.

\begin{lem}[\cite{Fr111}]\label{lem2.1}
Let $s \geq 2$ be an integer. Let $\mathcal{H}$ be an $r$-uniform graph on $[n]$ with $\nu(\mathcal{H}) < s$. Let $1 \leq i < j \leq n$ be a pair of vertices, then $\nu(S_{ij}(\mathcal{H})) < s$.
\end{lem}
\medskip

 Let $\cS(\cH)$ denote a shifted $r$-graph obtained from \(\mathcal{H}\) via a finite sequence of shifting operations. The following lemma demonstrates that the shifting operation does not decrease the spectral radius of an $r$-graph $\mathcal{H}$. To facilitate the proof, for any vertices $u, v \in V(\mathcal{H})$,  define $$E_u(\mathcal{H})=\{e \in E(\mathcal{H}): u \in e\}$$ and  $$E_{u \backslash v}(\mathcal{H})=\{e \in E(\mathcal{H}): u \in e, v \notin e\}.$$
\vspace{-3mm}
\begin{lem}\label{newlem}
For a given $r$-graph $\cH$ on $[n],$ we have $\rho(\cS(\cH))\geq \rho(\cH).$
\end{lem}
\begin{proof}[\bf Proof]

For a nonnegative eigenvector ${\bf x}=(x_{1},x_{2},\dots,x_{n})^\mathrm{T}$ corresponding to $\rho(\cH)$, relabel the vertices of $\mathcal H$ so that $x_1\ge x_2\ge \cdots \ge x_n$. We perform a sequence of shifting operations under this labeling:
\begin{equation*}
\cH=\cH_0\to \cH_1\to \cdots \to \cH_t=\cS(\cH).
\end{equation*}

Without loss of generality, suppose $\mathcal{H}_{k+1}=S_{ij}(\mathcal{H}_k)$ for some $0\le k\le t-1$ and $i<j\in[n]$. By $x_i\ge x_j$, one can check that
\begin{align*}
{\bf x}^\mathrm{T}\cA(\cH_{k+1}){\bf x} - {\bf x}^\mathrm{T}\cA(\cH_{k}){\bf x} &= \sum_{\{i_1,\dots,i_r\} \in E(\cH_{k+1})} r x_{i_1}x_{i_2}\cdots x_{i_r} - \sum_{\{i_1,\dots,i_r\} \in E(\cH_k)} r x_{i_1}x_{i_2}\cdots x_{i_r} \\[2mm]
&= \sum_{\{i_1,\dots,i_r\} \in E(S_{ij}(\cH_k))} r x_{i_1}x_{i_2}\cdots x_{i_r} - \sum_{\{i_1,\dots,i_r\} \in E(\cH_k)} r x_{i_1}x_{i_2}\cdots x_{i_r} \\[2mm]
&= \sum_{\substack{\{j,i_2,\dots,i_r\} \in E_{j\setminus i}(\cH_k) \\ \{i,i_2,\dots,i_r\} \notin E(\cH_k)}} r(x_i - x_j) x_{i_2}\cdots x_{i_r}\geq 0.
\end{align*}
Since the inequality ${\bf x}^\mathrm{T}\cA(\cH_{k+1}){\bf x}\geq {\bf x}^\mathrm{T}\cA(\cH_{k}){\bf x}$ holds for every $0\le k\le t-1$, by iterating over the full shifting sequence, we obtain
\begin{equation*}
{\bf x}^\mathrm{T}\cA(\cS(\cH)){\bf x}\ge {\bf x}^\mathrm{T}\cA(\cH){\bf x}.
\end{equation*}
Using Lemma \ref{lem:Qi}, it follows that
\begin{align*}
    \rho\big(\mathcal{S}(\mathcal{H})\big)
    =\max_{\substack{{\bf y}\ge \boldsymbol0\\  \|{\bf y}\|_r=1}}{\bf y}^\mathrm{T}\cA(\cS(\cH)){\bf y}
    \ge {\bf x}^\mathrm{T}\cA(\cS(\cH)){\bf x}
    \ge {\bf x}^\mathrm{T}\cA(\cH){\bf x}=\rho(\mathcal{H}).
\end{align*}
The result follows.
\end{proof}

\subsection{Auxiliary lemmas}
We begin this subsection by establishing an upper bound for the spectral radius of an $r$-graph in terms of its number of hyperedges.

\begin{lem}\label{lem:small-edge-spectral-bound}
Let $\cH$ be an $n$-vertex $r$-graph with $m$ hyperedges. Then $$\rho(\cH)\le r(r!)^{-1/r} m^{(r-1)/r}.$$
\end{lem}
\begin{proof}[\bf Proof]
By Lemma \ref{lem:Qi}, we have
\begin{equation*}
\rho(\cH)=\max_{\substack{{\bf x}\ge \boldsymbol0\\\|{\bf x}\|_r=1}} {\bf x}^\mathrm{T}\cA(\cH){\bf x}=\max_{\substack{{\bf x}\ge \boldsymbol0\\\|{\bf x}\|_r=1}} r\sum_{e\in E(\cH)}\prod_{v\in e}x_v.
\end{equation*}
For any nonnegative real vector ${\bf x}$, using H\"{o}lder's inequality, it follows that
\begin{equation*}
\sum_{e\in E(\cH)}\prod_{v\in e}x_v \le m^{(r-1)/r}\left(\sum_{e\in E(\cH)}\prod_{v\in e}x_v^r\right)^{1/r}.
\end{equation*}

Since $\sum_{v\in V(\cH)}x_v^r=1$, applying Maclaurin's inequality, we obtain
\begin{align*}
\sum_{e\in E(\cH)}\prod_{v\in e}x_v^r &\le \sum_{e\in\binom{V(\cH)}{r}}\prod_{v\in e}x_v^r \le \binom{n}{r} n^{-r} \le \frac{1}{r!}.
\end{align*}
Therefore,
$$\rho(\cH)\le r(r!)^{-1/r} m^{(r-1)/r}.$$
This finishes the proof.
\end{proof}

\medskip
\begin{lem}[\cite{Co-Du},  \cite{Ni}]\label{lem:subgraph}
Let $\mathcal{G}$ and $\mathcal{H}$ be $r$-graphs. If $\cG$ is a subgraph of $\cH$, then $\rho(\cG)\leq \rho(\cH).$
Moreover, if $\cH$ is connected, and $\cG$ is a proper subgraph of $\cH$, then $\rho(\cG) < \rho(\cH).$
\end{lem}
\medskip
Define
\begin{equation*}
c_{a,r} =
\begin{cases}
0, & \text{if } a = 0, \\
\frac{a^{(r-1)/r}}{(r-1)^{1/r}(r-2)!}, & \text{if } a \ge 1.
\end{cases}
\end{equation*}

Recall the definition of $\mathcal{F}_a(n)$. The following lemma gives an estimate for $\rho(\mathcal{F}_a(n)).$

\begin{lem}\label{lem:Fa-asymptotic}
For every fixed integer $a\ge 0$ and sufficiently large n, we have
\begin{equation*}
\rho(\mathcal{F}_a(n))=c_{a,r}n^{(r-1)^2/r}+O\left(n^{(r-2)(r-1)/r}\right).
\end{equation*}
\end{lem}
\begin{proof}[\bf Proof]
If $a=0$, then $\mathcal{F}_0(n)=\emptyset$ and the conclusion is immediate. Hence we may assume that $a\ge 1$.
Let $\mathcal{S}_a(n)$ be the subgraph of $\mathcal{F}_a(n)$ consisting of all hyperedges intersecting $[a]$ in exactly one vertex.

We begin by estimating $\rho\left(\mathcal{S}_a(n)\right)$.
Since $n-a\geq r-1$, by the definition of $\mathcal{F}_a(n)$, $\mathcal{S}_a(n)$ is connected. Let ${\bf x}$ be the Perron vector corresponding to $\rho(\mathcal{S}_a(n))$. Observe that any permutation of the vertices within $\{1, \dots, a\}$ or within $\{a+1, \dots, n\}$ preserves the edge set of $\mathcal{S}_a(n)$. Since the positive Perron vector ${\bf x}$ is unique up to scaling, its components must be invariant under these automorphisms. Consequently, we have $x_1 = \cdots = x_a$ and $x_{a+1} = \cdots = x_n$. Set
\begin{align*}
x_i = \alpha  \text{ for } 1\le i\le a, \quad \text{and}\quad x_j = \beta,\text{ for } a+1\le j\le n.
\end{align*}
From the eigenequations for $\rho(\mathcal{S}_a(n))$ at the vertices $x_i$, where $1\le i\le a$, it follows that
\[\rho(\mathcal{S}_a(n))\alpha^{r-1}=\binom{n-a}{r-1}\beta^{r-1}.\label{111}\]
From the eigenequations for $\rho(\mathcal{S}_a(n))$ at the vertices $x_j$, where $a+1\le j\le n$, then
\[\rho(\mathcal{S}_a(n))\beta^{r-1}=a\binom{n-a-1}{r-2}\alpha\beta^{r-2}.\label{222}\]
From (\ref{222}), we have
\begin{equation*}
\frac{\alpha}{\beta} = \frac{\rho(\mathcal{S}_a(n))}{a\binom{n-a-1}{r-2}}.
\end{equation*}
Substituting this into (\ref{111}), we obtain
\begin{equation}
\rho(\mathcal{S}_a(n))=\frac{a^{(r-1)/r}}{(r-1)^{1/r}(r-2)!}n^{(r-1)^2/r}+O\left(n^{(r-1)^2/r-1}\right).\label{jia}
\end{equation}

Now we estimate $\rho\left(\mathcal{F}_a(n)\right)$.
Recall that $\mathcal{S}_a(n)\subseteq\mathcal{F}_a(n)$. Let $\cD_a=\mathcal{F}_a(n)\setminus\mathcal{S}_a(n)$. Every hyperedge of $\cD_a$ intersects $[a]$ in at least two vertices, so $e(\cD_a)=O(n^{r-2})$. By Lemma~\ref{lem:small-edge-spectral-bound},
\begin{equation*}
\rho(\cD_a)=O\left(n^{(r-2)(r-1)/r}\right).
\end{equation*}
Since $\mathcal{S}_a(n)$ is a subgraph of $\mathcal{F}_a(n)$, Lemma~\ref{lem:subgraph} gives
\begin{align}
    \rho(\mathcal{F}_a(n)) \ge \rho(\mathcal{S}_a(n)),\label{ll}
\end{align}
which establishes the lower bound. For the upper bound, let ${\bf y}=(y_v)_{v\in [n]}$ be the Perron vector corresponding to $\rho(\mathcal{F}_a(n))$.  A further calculation shows that
\begin{align}
\rho(\mathcal{F}_a(n))&={\bf y}^\mathrm{T}\cA(\mathcal{F}_a(n)){\bf y} \notag\\[2mm]
&=r\sum_{e\in E(\mathcal{F}_a(n))}\prod_{v\in e}y_v \notag\\[2mm]
&= r\sum_{e\in E(\mathcal{S}_a(n))}\prod_{v\in e}y_v+r\sum_{e\in E(\cD_a)}\prod_{v\in e}y_v \notag\\[2mm]
&= {\bf y}^\mathrm{T}\cA(\mathcal{S}_a(n)){\bf y}+{\bf y}^\mathrm{T}\cA(\cD_a){\bf y} \leq \rho(\mathcal{S}_a(n))+\rho(\cD_a) \notag\\[2mm]
&= \frac{a^{(r-1)/r}}{(r-1)^{1/r}(r-2)!}n^{(r-1)^2/r}+O\left(n^{(r-2)(r-1)/r}\right).\label{uu}
\end{align}
Combining (\ref{jia}), (\ref{ll}) and (\ref{uu}), the desired equality follows.
\end{proof}


We write $\mathcal{G} \cup \mathcal{F}$ for the \textit{union} of two $r$-graphs $\mathcal{G}$ and $\mathcal{F}$, an $r$-graph with vertex set $V(\mathcal{G}) \cup V(\mathcal{F})$ and edge set $E(\mathcal{G}) \cup E(\mathcal{F})$.
\begin{lem}\label{lem:Fa-comparison}
Let $0\le a<s-1$ and $K:=K(s,r)>0$. Suppose $\cG$ is an $r$-graph satisfying $e(\cG)\le K n^{r-2}$. Then for sufficiently large $n$, we have
\begin{equation*}
\rho(\mathcal{F}_a(n)\cup \cG)<\rho(\mathcal{F}_{s-1}(n)).
\end{equation*}
\end{lem}
\begin{proof}[\bf Proof]
By Lemma~\ref{lem:small-edge-spectral-bound}, we obtain
$$
\rho(\cG)=O\left(n^{(r-2)(r-1)/r}\right).
$$
By Lemma~\ref{lem:Fa-asymptotic}, it gives that
$$\rho(\mathcal{F}_a(n))= c_{a,r}n^{(r-1)^2/r}+O\left(n^{(r-2)(r-1)/r}\right),$$
and
$$\rho(\mathcal{F}_{s-1}(n))= c_{s-1,r}n^{(r-1)^2/r}+O\left(n^{(r-2)(r-1)/r}\right).$$
Since $a<s-1$, we have $c_{a,r}<c_{s-1,r}$. Therefore, for all sufficiently large $n$,
$$
\rho(\mathcal{F}_a(n)\cup \cG)\le \rho(\mathcal{F}_a(n))+\rho(\cG)<\rho(\mathcal{F}_{s-1}(n)).
$$
This completes the proof.
\end{proof}

\section{Proof of Theorem~\ref{thm1}}\label{sec:proofs}
In the section, we give the proof of Theorem \ref{thm1}. The sketch of our proof is as follows. First, we reduce the problem to shifted $r$-graphs by applying the shifting operation, which does not increase the matching number and does not decrease the spectral radius.
Then we establish several structural properties of shifted (respectively, shifted-saturated) hypergraphs whose spectral radius is at least $\rho(\mathcal{F}_{s-1}(n))$ (Lemmas~\ref{lem:special-edge}--\ref{lem:shifted-saturated-extremal}). Using these properties, we prove the shifted version of Theorem~\ref{thm1} (Theorem~\ref{thm:shifted-spectral-extremal}). Finally, we remove the shifted assumption and derive Theorem~\ref{thm1} from the shifted version.

We now present several lemmas to establish Lemma \ref{lem:saturation-step}.

\medskip
\begin{lem}\label{lem:special-edge}
Let $\cH$ be a shifted $r$-graph on $[n]$ with $\nu(\cH)<s$. Suppose that $\rho(\cH)\ge \rho(\mathcal{F}_{s-1}(n))$ and $\mathcal{F}_a(n)\subseteq \cH$ for some $0\le a<s-1$. Set $t:=s-a$. Then for all sufficiently large $n$,
\begin{equation*}
e_a:=\{a+1, a+(t-1)r+2, a+(t-1)r+3,\ldots, a+tr\}\in E(\cH).
\end{equation*}
\end{lem}
\begin{proof}[\bf Proof]
Suppose to the contrary that $e_a \notin E(\cH)$.
We first deduce that no hyperedge in $E(\cH) \setminus E(\mathcal{F}_a(n))$ contains at least $r-1$ vertices from $[a+(t-1)r+2, n]$. If such a hyperedge $e$ exists, let its $r-1$ vertices in this interval be $y_1 < \dots < y_{r-1}$. Since $e \notin E(\mathcal{F}_a(n))$, the minimum labeled vertex of $e$ is at least $a+1$. Furthermore, $y_j \ge a+(t-1)r+1+j$ for every $j \in [r-1]$. Arranging the vertices of $e$ in increasing order yields $e_a \preceq e$. As $\cH$ is shifted, this implies $e_a \in E(\cH)$, a contradiction.

Since every hyperedge of $E(\cH)\setminus E(\mathcal{F}_a(n))$ contains at most $r-2$ vertices from $[a+(t-1)r+2,n]$, every hyperedge of $E(\cH)\setminus E(\mathcal{F}_a(n))$ intersects $I_a:=[a+1,\ a+(t-1)r+1]$ in at least two vertices. Hence $$\cH\subseteq \mathcal{F}_a(n)\cup \cG_a,$$
where
$$\cG_a=\left\{e\in\binom{[a+1,n]}{r}: |e\cap I_a|\ge 2\right\}.$$
By Lemma \ref{lem:subgraph},
\begin{equation*}
\rho(\cH)\le \rho(\mathcal{F}_a(n)\cup \cG_a).
\end{equation*}

Since $|I_a|=(t-1)r+1=O(1)$, we have $e(\cG_a)=O(n^{r-2}).$
By Lemma~\ref{lem:Fa-comparison} applied to $\cG_a$, we obtain
$$\rho(\cH)\le \rho(\mathcal{F}_a(n)\cup \cG_a)<\rho(\mathcal{F}_{s-1}(n)),$$
which contradicts $\rho(\cH)\ge \rho(\mathcal{F}_{s-1}(n))$. Therefore $e_a\in E(\cH)$.
\end{proof}

\begin{lem}\label{lem:compressed-matching}
Let $\cH$ be a shifted $r$-graph on $[n]$. Suppose $M_q^r\subseteq \cH$ and all hyperedges of $M_q^r$ are disjoint from $[m]$ for a given integer $m$. Then $\cH$ contains a copy of $M_q^r$ inside $[m+1,m+qr]$.
\end{lem}
\begin{proof}[\bf Proof]

Let $M$ be a copy of $M_q^r$ in $\cH$ with edge set $E(M)=\{e_1,\dots,e_q\}$. Enumerate the vertices in $\bigcup_{j=1}^q e_j$ in increasing order as
$$
u_1<u_2<\cdots<u_{qr}.
$$
Clearly, for every $i\in[qr]$, we have $u_i\ge m+i$ as $M$ is disjoint from $[m]$.

Since $e_j \subset\left\{u_1, \ldots, u_{q r}\right\}$ for each $j \in[q]$, its vertices are uniquely determined by their indices in $[q r]$. Let $I_j=\left\{i \in[q r]: u_i \in e_j\right\}$. Writing $I_j=\left\{i_{j, 1}, \ldots, i_{j, r}\right\}$ with $i_{j, 1}<\cdots<i_{j, r}$, we define
$$f_j=\left\{m+k: k \in I_j\right\} .$$
Since $I_1, \ldots, I_q$ form a partition of $[q r]$, it follows that $f_1, \ldots, f_q$ are pairwise disjoint $r$-subsets of $[m+1, m+q r]$.

It remains to show that $f_j\in E(\cH)$ for every $j\in[q]$. Writing the elements of $f_j$ and $e_j$ in increasing order gives
$$
f_j=\{m+i_{j,1}, \dots, m+i_{j,r}\}
$$
and
$$
e_j=\{u_{i_{j,1}}, \dots, u_{i_{j,r}}\}.
$$
By the fact that $u_i\ge m+i$ for all $i\in[qr]$, we have $m+i_{j,\ell}\le u_{i_{j,\ell}}$
for every $\ell\in[r]$. Hence $f_j\preceq e_j$. As $\cH$ is shifted and $e_j\in E(\cH)$, we obtain $f_j\in E(\cH)$.
Therefore, $f_1,\dots,f_q$ form a copy of $M_q^r$ in $\cH$ whose vertices are contained in $[m+1,m+qr]$.
\end{proof}

A shifted $r$-graph $\cH\subseteq\binom{[n]}{r}$ with $\nu(\cH)<s$ is called \textit{shifted-saturated} if there is no shifted $r$-graph $\cH'$ such that
$$ \cH\subsetneq \cH'\subseteq\binom{[n]}{r}\quad\text{and}\quad \nu(\cH')<s. $$
Namely, a shifted-saturated $r$-graph $\cH$ is an edge-maximal shifted graph (i.e., no proper shifted supergraph of $H$ has matching number less than $s$.).
Clearly, every shifted $r$-graph $\cH$ with $\nu(\cH)<s$ is contained in some shifted-saturated $r$-graph.

\begin{lem}\label{lem:saturation-step}
Let $\cH$ be a shifted-saturated $r$-graph on $[n]$ with $\nu(\cH)<s$. Suppose $\rho(\cH)\ge \rho(\mathcal{F}_{s-1}(n))$ and $\mathcal{F}_a(n)\subseteq \cH$ for some $0\le a<s-1$. Then for all sufficiently large $n$, we have $\mathcal{F}_{a+1}(n)\subseteq \cH.$
\end{lem}
\begin{proof}[\bf Proof]
For convenience, set $t:=s-a$. By Lemma~\ref{lem:special-edge},
\begin{equation*}
e_a:=\{a+1, a+(t-1)r+2, \ldots, a+tr\}\in E(\cH).
\end{equation*}
Suppose to the contrary that $\mathcal{F}_{a+1}(n)\nsubseteq \cH$. Since both $\cH$ and $\mathcal{F}_{a+1}(n)$ are shifted, the $r$-graph $\cH'=\cH\cup\mathcal{F}_{a+1}(n)$ is shifted and $\cH\subsetneq \cH'$. Since $\cH$ is shifted-saturated, we have $\nu(\cH')\ge s$. Thus $\cH'$ contains a copy $M$ of $M_s^r$.  Next, we will show $\nu(\cH)\ge s$, which leads to a contradiction.

Since $\nu(\cH)<s$, at least one hyperedge of $M$ lies in $E(\cH')\setminus E(\cH)$. Since $\mathcal{F}_a(n)\subseteq \cH$, every hyperedge in $E(\cH')\setminus E(\cH)$ belongs to $E(\mathcal{F}_{a+1}(n))\setminus E(\mathcal{F}_a(n))$, and consequently is disjoint from $[a]$ and contains $a+1$. Therefore, a matching can contain at most one such new hyperedge. It follows that $M$ contains exactly one hyperedge from $E(\cH')\setminus E(\cH)$, denote by $f$. The remaining $s-1$ hyperedges of $M$ belong to $\cH$ and are disjoint from $f$.
Among these $s-1$ hyperedges, at most $a$ intersect $[a]$. Hence at least $s-1-a=t-1$ of them are disjoint from $[a]$. Since they are also disjoint from $f$, and $a+1\in f$, these $t-1$ hyperedges are disjoint from $[a+1]$. By Lemma~\ref{lem:compressed-matching}, $\cH$ contains a copy of $M_{t-1}^r$, denoted by $M^*$,  inside $[a+2,\ a+(t-1)r+1]$. This matching is disjoint from $e_a$, since $e_a$ consists of $a+1$ together with $r-1$ vertices from $[a+(t-1)r+2,\ a+tr]$.

To complete the $s$-matching, it remains to find $a$ additional disjoint hyperedges by exploiting the assumption $\mathcal{F}_a(n)\subseteq \cH$. For sufficiently large $n$ and each $i\in[a]$, we select pairwise disjoint $(r-1)$-subsets $Y_i\subseteq [a+tr+1,n]$ and define $f_i=\{i\}\cup Y_i$. Since $f_i$ intersects $[a]$ at vertex $i$, it immediately follows that $f_i\in\mathcal{F}_a(n)\subseteq \cH$. By construction, the hyperedges $f_1,\dots,f_a$ are pairwise disjoint and are also vertex-disjoint from $M^*\cup\{e_a\}$. Consequently, $M^*\cup\{e_a,f_1,\dots,f_a\}$ forms a matching of size $(t-1)+1+a=s$ in $\cH$, which contradicts the condition $\nu(\cH)<s$. Hence, $\mathcal{F}_{a+1}(n)\subseteq \cH$, completing the proof.
\end{proof}

\medskip
\begin{lem}\label{lem:shifted-saturated-extremal}
Let $\cH^{sat}$ be a shifted-saturated $r$-graph on $[n]$ with $\nu(\cH^{sat})<s$. If $\rho(\cH^{sat})\ge \rho(\mathcal{F}_{s-1}(n))$, then $\cH^{sat}=\mathcal{F}_{s-1}(n)$.
\end{lem}
\begin{proof}[\bf Proof]
By repeatedly applying Lemma~\ref{lem:saturation-step}, we obtain $\mathcal{F}_{s-1}(n)\subseteq \cH^{sat}$. To establish the reverse inclusion $\cH^{sat} \subseteq \mathcal{F}_{s-1}(n)$, we show that $\cH^{sat}$ has no hyperedge contained entirely in $[s,n]$. Suppose for contradiction that there exists $e\in E(\cH^{sat})$ with $e\subseteq [s,n]$. For sufficiently large $n$, we can choose pairwise disjoint $(r-1)$-subsets $Y_i\subseteq [s,n]\setminus e$ for each $i\in[s-1]$, and define $e_i=\{i\}\cup Y_i$. Since $e_i$ intersects $[s-1]$ at $i$, we have $e_i\in \mathcal{F}_{s-1}(n)\subseteq \cH^{sat}$. By construction, the hyperedges $e_1,\dots,e_{s-1}$ are pairwise disjoint and also disjoint from $e$. Hence, $\{e,e_1,\dots,e_{s-1}\}$ forms a matching of size $s$ in $\cH^{sat}$, contradicting $\nu(\cH^{sat})<s$. Thus, every hyperedge of $\cH^{sat}$ must intersect $[s-1]$, which implies $\cH^{sat} \subseteq \mathcal{F}_{s-1}(n)$. Combined with $\mathcal{F}_{s-1}(n)\subseteq \cH^{sat}$, we conclude that $\cH^{sat}=\mathcal{F}_{s-1}(n)$.
\end{proof}
\medskip
\begin{thm}\label{thm:shifted-spectral-extremal}
Let $\cH$ be a shifted $r$-graph on $[n]$ with $\nu(\cH)<s$. For sufficiently large $n$, we have
\begin{equation*}
\rho(\cH)\le \rho(\mathcal{F}_{s-1}(n)).
\end{equation*}
Moreover, equality holds if and only if $\cH=\mathcal{F}_{s-1}(n)$.
\end{thm}
\begin{proof}[\bf Proof]
We first show $\rho(\cH)\le \rho(\mathcal{F}_{s-1}(n))$. Clearly, $\cH \subseteq \cH^{sat}$ for some shifted-saturated $r$-graph $\cH^{sat}$ with $\nu(\cH^{sat})<s$. By Lemma~\ref{lem:subgraph}, we obtain $\rho(\cH)\le \rho(\cH^{sat})$. If $\rho(\cH)>\rho(\mathcal{F}_{s-1}(n))$, then $\rho(\cH^{sat})>\rho(\mathcal{F}_{s-1}(n))$. By Lemma~\ref{lem:shifted-saturated-extremal}, this gives $\cH^{sat}=\mathcal{F}_{s-1}(n)$ and so
$$
\rho(\cH)>\rho(\mathcal{F}_{s-1}(n))=\rho(\cH^{sat}),
$$
a contradiction. Hence, $\rho(\cH)\le \rho(\mathcal{F}_{s-1}(n))$.

Now, we characterize the unique extremal hypergraph $\cH$ satisfying $\rho(\cH)=\rho(\mathcal{F}_{s-1}(n))$. For the saturated extension $\cH^{sat}$ of $\cH$, we have
$$
\rho(\cH^{sat})\ge \rho(\cH)=\rho(\mathcal{F}_{s-1}(n)).
$$
By Lemma~\ref{lem:shifted-saturated-extremal}, this implies
$$
\rho(\cH^{sat})= \rho(\cH)=\rho(\mathcal{F}_{s-1}(n)).
$$
Since the hypergraph $\mathcal{F}_{s-1}(n)$ is connected for $s\ge 2$, $\rho(\cH^{sat})\ge \rho(\cH)$ with equality if and only if $\cH=\cH^{sat}$. Consequently, $\cH=\mathcal{F}_{s-1}(n)$, completing the proof.
\end{proof}
\medskip
Next, we complete the proof of Theorem~\ref{thm1} using the following lemma.

\begin{lem}[\cite{WaP}]\label{thm1.2}
Let $n \ge 2r+s-2$. If an $r$-graph $\cH \subseteq \binom{[n]}{r}$ is not isomorphic to $\cF_{s-1}(n)$ but satisfies $S_{ij}(\cH) = \cF_{s-1}(n)$, then $\nu(\cH) \ge s$.
\end{lem}

\begin{proof}[\bf Proof of Theorem~\ref{thm1}]
Let $\cG$ be any extremal $r$-graph of Theorem \ref{thm1}. By repeatedly applying the shifting operation, we obtain a shifted $r$-graph $\cH$. Lemmas~\ref{lem2.1} and \ref{newlem} ensure that the shifting operation preserves the matching condition and does not decrease the spectral radius. Thus, $\cH$ is also an extremal hypergraph.

Theorem~\ref{thm:shifted-spectral-extremal} establishes that, for sufficiently large $n$, $\cF_{s-1}(n)$ uniquely maximizes the spectral radius among all shifted $r$-graphs with $\nu(\cH)<s$. This implies $\cH = \cF_{s-1}(n)$. Lemma~\ref{thm1.2} then guarantees that
$$\cG\cong\cH=\cF_{s-1}(n).$$
This completes the proof of Theorem~\ref{thm1}.
\end{proof}

\section*{Acknowledgments}
Research of  Kang and Lu was supported by the National Natural Science Foundation of China (grant numbers 12331012, 12571375).
The research of Yuan and Zhou was supported by the National Natural Science Foundation of China (grant numbers~12271337, 12371347).

\section*{Declaration of interest}
The authors declare no known conflicts of interest.




\begin{thebibliography}{99}
\small{



\bibitem{68} B. Bollob\'{a}s, D. Daykin, P. Erd\H{o}s, Sets of independent edges of a hypergraph, \textit{Quarterly Journal of Mathematics: Oxford Series (2)}, \textbf{27}(105) (1976) 25--32.

\bibitem{Co-Du} J. Cooper, A. Dutle, Spectra of uniform hypergraphs, \textit{Linear Algebra Appl.} \textbf{436}(9) (2012) 3268--3292.

\bibitem{Er1} P. Erd\H{o}s, A problem on independent $r$-tuples, \textit{Ann. Univ. Sci. Budapest}, \textbf{8} (1965) 93--95.

\bibitem{EGa} P. Erd\H{o}s, T. Gallai, On maximal paths and circuits of graphs, \textit{Acta Math. Acad. Sci. Hung.} \textbf{10} (1959) 337--356.

\bibitem{EKR} P. Erd\H{o}s, C. Ko, R. Rado, Intersection theorems for systems of finite sets, \textit{Quart. J. Math. Oxford Ser.} \textbf{12}(2) (1961) 313--320.

\bibitem{Fr111} P. Frankl, The shifting technique in extremal set theory, \textit{Surv. Combin.} \textbf{123} (1987) 81--110.

\bibitem{Fr1} P. Frankl, Improved bounds for Erd\H{o}s' matching conjecture, \textit{J. Combin. Theory Ser. A} \textbf{120} (2013) 1068--1072.

\bibitem{Fran2017} P. Frankl, On the maximum number of edges in a hypergraph with a given matching number, \textit{Discrete Appl. Math.} \textbf{216} (2017) 562--581.

\bibitem{213} P. Frankl, Proof of the Erd\H{o}s Matching Conjecture in a new range, \textit{Israel J. Math.} \textbf{222}(1) (2017) 421--430.

\bibitem{FrKu} P. Frankl, A. Kupavskii, The Erd\H{o}s matching conjecture and concentration inequalities, \textit{J. Combin. Theory Ser. B} \textbf{157} (2022) 366--400.

\bibitem{215} P. Frankl, T. \L uczak, K. Mieczkowska, On matchings in hypergraphs, \textit{Electron. J. Combin.} \textbf{19}(2) (2012) P42.

\bibitem{Fr-Ga-Ha} S. Friedland, S. Gaubert, L. Han, Perron-Frobenius theorem for nonnegative multilinear forms and extensions, \textit{Linear Algebra Appl.} \textbf{438} (2013) 738--749.



\bibitem{Hou-Hu-Liu} J. Hou, C. Hu, X. Liu, A finite-board reduction for the Erd\H{o}s Matching Conjecture and the 4-uniform case via exact certificates, \textit{arXiv preprint}, arXiv:2605.26060, 2026.

\bibitem{312} H. Huang, P. Loh, B. Sudakov, The size of a hypergraph and its matching number, \textit{Combin. Probab. Comput.} \textbf{21}(03) (2012) 442--450.


\bibitem{KoKu} D. Kolupaev, A. Kupavskii, Erd\H{o}s' Matching Conjecture for almost perfect matchings, \textit{Discrete Math.} \textbf{346} (2023) 113304.

\bibitem{KuSo} A. Kupavskii, G. Sokolov, A complete solution of the Erd\H{o}s-Kleitman matching problem for $n\leq 3s$, \textit{arXiv preprint}, arXiv:2511.21628, 2025.

\bibitem{Lim} L. Lim, Singular values and eigenvalues of tensors: a variational approach, in: \textit{Proceedings of the IEEE International Workshop on Computational Advances in Multi-Sensor Adaptive Processing (CAMSAP'05)}, \textbf{1} (2005) 129--132.


\bibitem{LuMi} T. \L uczak, K. Mieczkowska, On Erd\H{o}s' extremal problem on matchings in hypergraphs, \textit{J. Combin. Theory Ser. A} \textbf{124} (2014) 178--194.


\bibitem{Ni} V. Nikiforov, Analytic methods for uniform hypergraphs, \textit{Linear Algebra Appl.} \textbf{457} (2014) 455--535.

\bibitem{Qi-new} L. Qi, Eigenvalues of a real supersymmetric tensor, \textit{J. Symbolic Comput.} \textbf{40}(6) (2005) 1302--1324.

\bibitem{Qi} L. Qi, Symmetric nonnegative tensors and copositive tensors, \textit{Linear Algebra Appl.} \textbf{439} (2013) 228--238.

\bibitem{Sh} J. Shao, A general product of tensors with applications, \textit{Linear Algebra Appl.} \textbf{439} (2013) 2350--2366.


\bibitem{WaP} W. Wang, Y. Peng, Counting the maximum number of sunflowers in hypergraphs with given matching number, \textit{European J. Combin.} \textbf{136} (2026) 104379.

\bibitem{Ya-Ya} Y. Yang, Q. Yang, Further results for Perron-Frobenius theorem for nonnegative tensors, \textit{SIAM J. Matrix Anal. Appl.} \textbf{31}(5) (2010) 2517--2530.

}

\end{thebibliography}
\end{document}